\documentclass[12pt,a4paper]{article}
\usepackage[utf8]{inputenc}
\usepackage[english]{babel}
\usepackage{amsmath}
\usepackage{hhline}
\usepackage{ragged2e}
\usepackage{fancyhdr}
\usepackage{authblk}

\usepackage[margin=2.3cm]{geometry}
\usepackage{parskip}
\usepackage{bbm}
\usepackage{mathrsfs}
\usepackage{amsfonts}
\usepackage{placeins}
\usepackage{graphicx}
\usepackage{subcaption}
\usepackage{hyperref}
\usepackage{tabularx}
\usepackage{caption}
\usepackage{float}
\usepackage{mathtools}
\usepackage{listings}
\usepackage{xcolor}
\usepackage{systeme}
\usepackage{tabularx}
\usepackage{caption}
\usepackage{diagbox}
\usepackage{longtable}
\usepackage{booktabs}
\usepackage{multirow}
\usepackage{siunitx}
\usepackage{enumerate}
\usepackage{amssymb}
\usepackage{amsthm}

\usepackage[square,numbers]{natbib}

\newtheorem{theorem}{Theorem}[section]
\newtheorem{lemma}{Lemma}[section]

\newtheorem{definition}{Definition}[section]
\newtheorem{proposition}{Proposition}[section]

\newcommand{\E}{\mathbb{E}}
\newcommand{\R}{\mathbb{R}}
\newcommand{\I}{\mathbb{I}}
\renewcommand{\P}{\mathbb{P}}
\newcommand{\N}{\mathbb{N}}

\providecommand{\keywords}[1]{\small\textbf{\textbf{Keywords:}} #1}

\providecommand{\MSC}[1]{\small \textbf{\textbf{MSC 2020:}} #1}

\title{Weak convergence of stochastic integrals with applications to SPDEs}
\date{}
\author[]{Xavier Bardina\thanks{Corresponding author. \\Both authors are supported by the grant PID2021-123733NB-I00 from SEIDI, Ministerio de Econom\'ia y Competitividad.}}
\author[]{Salim Boukfal}

\affil[]{Departament de Matemàtiques, Universitat Autònoma de Barcelona, Cerdanyola del Vallès, Spain}

\affil[]{xavier.bardina, salim.boukfal@uab.cat}

\begin{document}

\maketitle

\begin{abstract}
    In this paper we provide sufficient conditions for sequences of random fields of the form $\int_{D} f(x,y) \theta_n(y) dy$ to weakly converge, in the space of continuous functions over $D$, to integrals with respect to the Brownian sheet, $\int_{D} f(x,y)W(dy)$, where $D \subset \R^d$ is a rectangular domain, $x \in D$, $f$ is a function satisfying some integrability conditions and $\{\theta_n\}_n$ is a sequence of stochastic processes whose integrals $\int_{[0,x]}\theta_n(y)dy$ converge in law to the Brownian sheet (in the sense of the finite dimensional distribution convergence). We then apply these results to stablish the weak convergence of solutions of the stochastic Poisson equation. 
\end{abstract}


\MSC{60F05, 60F17, 60G50, 60G60, 60H05, 60H15}

\section{Introduction}

Let $d \geq 1$ be any integer and $T = (T_1,...,T_d) \in \R_+^d$ be fixed and consider in $D = [0,T] = \prod_{i=1}^d [0,T_i] = [0,T_1] \times ... \times [0,T_d]$ the usual partial ordering.

In, for instance, \cite[Corollary 1, p. 683]{wichurarandomwalk} or \cite[Theorem A.1]{bardinaboukfal1}, and \cite{multiparameterpoisson}, the authors exhibited examples of sequences of processes $\zeta_n = \{\zeta_n(x) \colon x \in D \}$, converging weakly to the Brownian sheet in the space of continuous functions over $D$, $\mathcal{C}(D)$, generalizing the well-known Donsker's invariance principle and the results proved by Stroock in \cite{stroock1982lectures} in the multidimensional parameter set case, respectively.

A natural question to ask ourselves then is if the integrals with respect to such processes converge, in the same sense, as stochastic processes, towards integrals with respect to the Brownian sheet. Of course, this problem has already been addressed in \cite{articlefrances}, \cite{kurtzprotter} (when a single parameter is considered) and, partially, in \cite{bardinaheat}, where two parameters are considered and the convergence of the stochastic integrals of a particular integrand (the Green function associated to the heat equation with Dirichlet boundary conditions) is then used to show the convergence of solutions of the stochastic heat equation. Mostly motivated by the latter, in \cite{bardinaboukfal1}, the authors consider sequences of kernels $\{\theta_n\}_{n \in \mathbb{N}}$ so that the processes $\zeta_n$ can be written as $\zeta_n(x) = \int_{[0,x]}\theta_n(y)dy$ and provide sufficient conditions for the processes
\begin{equation*}
    X_n \coloneqq \left\{ X_n(x) = \int_{[0,x]} f(y)\theta_n(y)dy \colon x \in D \right\}
\end{equation*}
to weakly converge, in $\mathcal{C}(D)$, to the process
\begin{equation*}
    X \coloneqq \left\{ X(x) = \int_{[0,x]} f(y)W(dy) \colon x \in D \right\}.
\end{equation*}
Where $W = \{W(x) \colon x \in D\}$ denotes a Brownian sheet and $f$ is some function in $L^{2q}(D)$ for some real $q \geq 1$. 

However, in many applications (for instance, in stochastic partial differential equations) the integrands involved might depend on the time/space parameter as well, which motivates the study of the same problem when the integrand depends on it. Hence, the aim of this paper is to continue with the study of the weak approximation of stochastic integrals in \cite{bardinaboukfal1} to extend the results to integrals of the form
\begin{equation*}
    \int_{[0,x]} f(x,y)W(dy) \quad \text{and}\quad \int_{D} f(x,y)W(dy),
\end{equation*}
where $f \colon D \times D \to \R$ will be some function satisfying some integrability conditions. Given that
\begin{equation*}
    \int_{[0,x]} f(x,y)W(dy) = \int_{D} \I_{[0,x]}(y) f(x,y)W(dy),
\end{equation*}
where $\I_A$ denotes the indicator function of $A \subset \R^d$ ($\I_A (y) = 1$ if $y \in A$ and $\I_A(y) = 0$ otherwise), we will focus our efforts on studying the latter. We will see, as well, how these results can be applied to prove the weak convergence of solutions to the stochastic Poisson equation with Dirichlet boundary conditions.

The paper is organized as follows, Section \ref{sec:preliminaries} is devoted to briefly introduce the Brownian sheet, the stochastic integral with respect to it and to provide some results regarding the weak convergence of stochastic processes. In Section \ref{sec:main result} we state and proof the main theorem and, in Section \ref{sec:SPDEs} we see how these results can be applied to prove the weak convergence of solutions of the Poisson equation.

\section{Preliminaries}\label{sec:preliminaries}

In this section we shall provide the main definitions and tools we will be working with.

\subsection{Brownian sheet and integrals with respect to it}

To define the Brownian sheet and the stochastic integral with respect to such process, we will make use of the isonormal Gaussian process over a real separable Hilbert space $H$ with inner product $\langle \cdot , \cdot \rangle_H$.

\begin{definition}
    We say that a stochastic process $W = \{W(h) \colon h \in H\}$ defined in a complete probability space is an isonormal Gaussian process if it is a centered Gaussian process with covariance function $\text{Cov}(f,g) = \langle f, g \rangle_H$ for all $f,g \in H$. 
\end{definition}

From now on, we shall take $H=L^2(D)$ with the usual inner product. A Brownian sheet (or $d$-parameter Wiener process) is then defined as the process $\Tilde{W} = \{\Tilde{W}(x) \colon x \in D \}$ with
\begin{equation*}
    \Tilde{W}(x) = W\left(\I_{[0,x]}\right).
\end{equation*}

For a given function $f \in L^2(D)$ and $x \in D$, we then define the Wiener integral of $f$ with respect to the Brownian sheet over $[0,x]$ as $W(f \I_{[0,x]})$ and denote it by
\begin{equation*}\label{integral with respect to brownian sheet}
    W(f \I_{[0,x]}) = \int_{[0,x]} f(y) \Tilde{W}(dy).
\end{equation*}
To simplify the notation, we will write $\int_{[0,x]} f(y) {W}(dy)$ instead of $\int_{[0,x]} f(y) \Tilde{W}(dy)$.

One can easily check (via Kolmogorov's continuity Theorem), that the Brownian sheet and the integral of an $L^2(D)$ function with respect to it have a continuous version, so, when talking about these objects, we will be talking about the continuous versions.

\subsection{Weak convergence of stochastic processes}

To stablish the weak convergence (or convergence in law or in distribution) of a sequence of stochastic processes $\{X_n\}_{n \in \mathbb{N}}$ in $\mathcal{C}(D)$, one usually needs to prove the convergence of the finite dimensional distributions and that it is relatively compact or, by Prokhorov's Theorem, that it is tight. The first result we introduce (whose proof can be found in \cite[Lemma 2.1]{bardinaboukfal1}) is a result regarding the convergence in law of random variables, which will be useful when proving the convergence of the finite dimensional distributions, while the second one is a tightness criterion which generalizes Proposition 2.3, p. 95, in \cite{billingsleypla}.

\begin{lemma}\label{lema espais normats}
    Let $(F,||\cdot||)$ be a normed vector space and $\{J^n\}_{n \in \mathbb{N}}$ and $J$ linear maps from $F$ to $L^1(\Omega)$ such that
    \begin{equation*}
        \sup_{n\geq 1} \E\left[ \left| J^n(f)  \right| \right] \leq C ||f||
    \end{equation*}
    and
    \begin{equation*}
        \E\left[ \left| J(f)  \right| \right] \leq C ||f||
    \end{equation*}
    for any $f \in F$ and for some positive constant $C$ and assume there is a dense subset $D$ of $F$ with respect to the norm $||\cdot ||$ such that $J^n(f)$ converges in law to $J(f)$ as $n$ approaches infinity for any $f \in D$. Then, for any $f \in F$, $J^n(f)$ converges in law to $J(f)$ as $n$ approaches infinity.
\end{lemma}

\begin{proposition}\label{tightness criterion}
    Let $\{X_n\}_{n \in \mathbb{N}}$ be a sequence of random fields in $D$ whose sample paths vanish on $\{x \in D \colon x_1 \cdot ... \cdot x_d = 0\}$ almost surely and for which there are constants $C, p > 0$ and $\delta > d$ such that
    \begin{equation}\label{moment condition}
        \sup_{n} \E\left[ \left|X_n(x) - X_n(z)  \right|^{p} \right] \leq C \left(\sum_{i=1}^d |x_i - z_i| \right)^\delta.
    \end{equation}
    Then, for each $n \in \N$, $X_n$ possesses a continuous version and the sequence of continuous versions is tight in $\mathcal{C}(D)$.
\end{proposition}
\begin{proof}
    The existence of a continuous version follows from Kolmogorov's continuity Theorem. Hence, from now on, we shall consider that the processes $X_n$ have a.s. continuous sample paths. 
    
    In order to prove tightness, we will see that the hypotheses of \cite[Theorem 3, p. 1665]{bickelwichura} are fulfilled (when we consider the Lebesgue measure) under condition \eqref{moment condition}, that is, for fixed $x,z \in D$, $x \leq z$, we aim to prove that
    \begin{equation*}
        \sup_{n} \E\left[|\Delta_x X_n(z)|^p \right] \leq C' \left( \prod_{i=1}^d (z_i - x_i) \right)^{\delta'}
    \end{equation*}
    for some $C' > 0$ and $\delta' > 1$ and where $\Delta_x X_n(z)$ denotes the increment of $X_n$ over the rectangle $[x,z]$, which is given by
    \begin{equation*}
        \Delta_x X_n(z) = \sum_{\varepsilon_1, ..., \varepsilon_d \in \{0,1\}} (-1)^{d - \sum_{i=1}^d \varepsilon_i} X_n\left(x_1 + \varepsilon_1 (z_1 - x_1), ..., x_d + \varepsilon_d (z_d - x_d)\right).
    \end{equation*}
    Now fix any $n \in \N$, $\varepsilon_1,...,\varepsilon_d \in \{0,1\}$, $j \in \{1,...,d\}$ and consider the term
    \begin{equation*}
        (-1)^{d - \sum_{i=1}^d \varepsilon_i} X_n\left(x_1 + \varepsilon_1 (z_1 - x_1), ..., x_d + \varepsilon_d (z_d - x_d)\right).
    \end{equation*}
    For this term, there will be one (and only one) summand 
    \begin{equation*}
        (-1)^{d - \sum_{i=1}^d \varepsilon_i'} X_n\left(x_1 + \varepsilon_1' (z_1 - x_1), ..., x_d + \varepsilon_d' (z_d - x_d)\right)
    \end{equation*}
    in $\Delta_x X_n(z)$ such that $\varepsilon_i = \varepsilon_i'$ for all $i \in \{1,...,d\}\backslash \{j\}$ and $\varepsilon_j \neq \varepsilon_j'$. In particular, this will mean that $(-1)^{d - \sum_{i=1}^d \varepsilon_i} \cdot (-1)^{d - \sum_{i=1}^d \varepsilon_i'} = -1$. Hence, by using that for all $p \geq 0$ there is a constant $C_p$ ($C_p = 1$ if $p \leq 1$ and $C_p = 2^{p-1}$ if $p \geq 1$) such that $|a+b|^p \leq C_p(|a|^p + |b|^p)$ for any $a,b \in \R$ and by taking $j \in \{1,...,d\}$ so that $|z_j - x_j| = \min_{1\leq i \leq d}|z_i - x_i|$, we will have
    \begin{equation*}
        \E\left[ \left| \Delta_x X_n(z)  \right|^p  \right] \leq C_p C 2^{d-1} |z_j - x_j|^{\delta} \leq C' \left( \prod_{i=1}^d |z_i - x_i|\right)^{\delta'},
    \end{equation*}
    where $C' = C_p C 2^{d-1} > 0$ and $\delta' = \delta/d > 1$.
\end{proof}
The reader might ask why such result needs to be derived when, as mentioned in the proof of the proposition, we already have a tightness criterion, \cite[Theorem 3, p. 1665]{bickelwichura}. This is due to the fact that, in many situations, and as we will see in Section \ref{sec:SPDEs}, it is easier to deal with the differences $|f(x,\cdot) - f(z,\cdot)|$ rather than dealing with the increments in the first variable of $f(x,\cdot)$. 

\section{Statement and proof of the main result}\label{sec:main result}

This section is devoted to prove the convergence in distribution, in $\mathcal{C}(D)$, of the processes
\begin{equation}\label{aproximadors}
    X_n \coloneqq \left\{ X_n(x) = \int_{D} f(x,y)\theta_n(y)dy \colon x \in D \right\}
\end{equation}
towards the process
\begin{equation}\label{wiener integral}
    X \coloneqq \left\{ X(x) = \int_{D} f(x,y)W(dy) \colon x \in D \right\}.
\end{equation}
Of course, we will have to impose some conditions on the function $f \colon D \times D \to \R$ to make sense of integrals like $\int_{D} f(x,y)W(dy)$ and to guarantee the existence of a version of the processes $X_n$ and $X$ possessing a.s. continuous sample paths so that we can talk about convergence in $\mathcal{C}(D)$.

\begin{theorem}\label{main result}
    Let $f \colon D \times D \to \R$ be a measurable function such that, for fixed $z \in\{x \in D \colon x_1 \cdot ... \cdot x_d = 0\}$, $f(z,y) = 0$ for almost every $y \in D$ (with respect to the Lebesgue measure) and such that for each $x,z \in D$, $f(x,\cdot) \in L^{2q}(D)$ for some $q \geq 1$ and
    \begin{equation}\label{L2qHolder}
        ||f(x,\cdot) - f(z,\cdot) ||_{2q} \leq C ||x - z||^{\alpha}
    \end{equation}
    for some constants $C, \alpha > 0$ and where $|| \cdot ||_{2q}$ is the usual norm in $L^{2q}(D)$ and $||\cdot ||$ is some norm in $\R^d$. Suppose that, for any $g \in L^{2q}(D)$ and for some $m > d / \alpha $,
    \begin{equation}\label{fita moments}
        \sup_{n}\E\left[ \left| \int_{D}  g(y)\theta_n(y)dy \right|^m \right] \leq C ||g||_{2q}^m
    \end{equation}
    holds for some constant $C > 0$, that the processes $\zeta_n  = \{\int_{[0,x]}\theta_n(y)dy \colon x \in D\}$ converge towards the Brownian sheet (in the sense of convergence of the finite dimensional distributions) and that $\theta_n \in L^2(D)$ almost surely for each $n$. Then, the processes $X_n$ and $X$ defined by \eqref{aproximadors} and \eqref{wiener integral}, respectively, possess a continuous version and the former converges in law to the latter in $\mathcal{C}(D)$.
\end{theorem}
\begin{proof}
    We start by observing that, since $f(x,\cdot) \in L^{2q}(D)$ for any $x \in D$, we have, in particular, that $f(x,\cdot) \in L^2(D)$ and hence, the stochastic integral $\int_D f(x,y)W(dy)$ makes sense for each $x \in D$. Similarly, since $\theta_n$ is square integrable, the integrals defining $X_n$ are well defined for each $x \in D$ as well.

    \textbf{Tightness and existence of a continuous version}

    By \eqref{fita moments} and \eqref{L2qHolder}, we have that
    \begin{equation*}
        \E\left[ \left|X_n(x) - X_n(z)\right|^m \right] = \E\left[ \left| \int_D (f(x,y) - f(z,y)) \theta_n(y) dy   \right|^m \right] \leq C ||f(x,\cdot) - f(z,\cdot)||_{2q}^m \leq C ||x - z||^{m \alpha},
    \end{equation*}
    for some constant $C>0$ independent of $x,z \in D$ (which might have changed from one inequality to another) and where we have used that $f(x,\cdot), f(z,\cdot) \in L^{2q}(D)$ (so their difference is in $L^{2q}(D)$ as well). Given that all norms in $\R^d$ are equivalent, we obtain that
    \begin{equation*}
        \sup_n \E\left[ \left|X_n(x) - X_n(z)\right|^m \right] \leq C \left( \sum_{i=1}^m |x_i - y_i|  \right)^{m \alpha}.
    \end{equation*}
    Since $m \alpha > d$, Proposition \ref{tightness criterion} gives us the existence of a continuous version for each $n$ (which we shall denote by $X_n$ as well) and the tightness of the sequence. 
    
    To show that $X$ has a version with continuous sample paths, one only needs to recall that $\int_D \left(f(x,y) - f(z,y) \right) W(dy)$ is a normal random variable with variance
    \begin{equation*}
        ||f(x,\cdot) - f(z,\cdot)||_2^2 \leq C ||f(x,\cdot) - f(z,\cdot) ||_{2q}^2 \leq C ||x - z||^{2\alpha} \leq C \left(\sum_{i=1}^d |x_i - z_i|\right)^{2\alpha}.
    \end{equation*}
    Where we have used Hölder's inequality and, again, the fact that all norms in $\R^d$ are equivalent. Hence, for any even integer $N > d/\alpha$, we have that
    \begin{equation*}
        \E\left[ \left|X(x) - X(z)   \right|^N  \right] \leq C \left( \sum_{i=1}^d |t_i - s_i| \right)^{\alpha'}
    \end{equation*}
    for some positive constant $C$ (independent of $x$ and $z$) and for $\alpha' = N\alpha > d$.

    \textbf{Convergence of the finite dimensional distributions}

    By the Cramér-Wold device, it suffices to show that, for any $k \in \N$, $z^1, ..., z^k \in D$ and $a_1,...,a_k \in D$, $S_n$ converges in law towards $S$, where
    \begin{equation*}
        S_n = \sum_{j=1}^m a_j X_n(z^j), \quad S = \sum_{j=1}^m a_j X(z^j).
    \end{equation*}
    But observe that $S_n = J^n(K)$, $S = J(K)$, where
    \begin{equation*}
        J^n(g) = \int_D g(y) \theta_n(y) du, \quad J(g) = \int_D g(y) W(dy), \quad K(y) = \sum_{j=1}^m a_j f(z^j,y).
    \end{equation*}
    Given that $f(z,\cdot) \in L^{2q}(D)$ for each $z \in D$, we have that $K$ is in $L^{2q}(D)$ as well. Moreover, by \eqref{fita moments} and the isometry property of the Wiener integral (alongside with Hölder's inequality), we have that
    \begin{equation*}
        \sup_{n} \E\left[ \left| J^n(g)  \right| \right] \leq C ||g||_{2q}, \quad \E\left[ \left| J(g)  \right| \right] \leq C ||g||_{2q},
    \end{equation*}
    for some positive constant $C $ and for any $g \in L^{2q}(D)$. Hence, by Lemma \ref{lema espais normats} applied to $F = L^{2q}(D)$ with the usual norm, it suffices to show that $J^n(g)$ converges in law towards $J(g)$ for simple functions $g$ of the form
    \begin{equation*}
        g(y) = \sum_{j=1}^l g_j \I_{(x_{j-1}, x_j]} (y),
    \end{equation*}
    with $l \geq 1$, $g_j \in \R$ and $0 = x_0 < x_1 < ... < x_l =T $, which are dense in $L^{2q}(D)$. To this purpose, we simply note that, for such functions $g$,
    \begin{equation*}
        J^n(g) = \int_{D} \left( \sum_{j=1}^l g_j \I_{(x_{j-1}, x_j]} (y) \right) \theta_n (y)dy 
        = \sum_{j=1}^l g_j \int_{(x_{j-1}, x_j]} \theta_n(y)dy,
    \end{equation*}
    which converges in law towards
    \begin{equation*}
        \sum_{j=1}^l g_j \int_{(x_{j-1}, x_j]} W(dy) = \int_{D} \left( \sum_{j=1}^l g_j \I_{(x_{j-1}, x_j]} (y) \right) W(dy) = J(g)
    \end{equation*}
    because the finite dimensional distributions of $\zeta_n$ converge to those of the Brownian sheet.
\end{proof}

As in \cite{bardinaboukfal1}, condition \eqref{fita moments} implies that, for each $x \in D$, the sequence $\{(X_n(x))^2\}_{n \in \N}$ is uniformly integrable, so
\begin{equation*}
    \lim_{n \to \infty} \E\left[ \left( \int_D f(x,y)\theta_n(y)dy \right)^2 \right] = \int_D f^2(x,y)dy. 
\end{equation*}

Of course, otherwise this result would be meaningless, there are examples of sequences of kernels $\{\theta_n\}_n$ verifying the above hypotheses. This is the content of \cite[Sections 3.1 and 3.2]{bardinaboukfal1}, where the Donsker kernels:
\begin{equation}\label{donsker kernels}
    \theta_n(x) = n^{\frac{d}{2}} \sum_{k = (k_1,...,k_d) \in \mathbb{N}^d}Z_k \mathbb{I}_{[k-1,k)}(nx),
\end{equation}
being $\{Z_k\}_{k \in \N^d}$ a sequence of i.i.d. random variables with zero mean, unitary variance and finite moments of order $m$ (in this case, $\zeta_n$ would be the analog of the linear interpolation of a random walk in the multidimensional parameter set case), and the Kac-Stroock kernels:
\begin{equation*}
    \theta_n(x) = n^{\frac{d}{2}} \left( \prod_{i=1}^d x_i \right)^{\frac{d-1}{2}} (-1)^{N_n(x)},
\end{equation*}
where $N_n$ is a multiparameter Poisson process of intensity $n$, are considered.

\section{Application to SPDEs}\label{sec:SPDEs}

In this section we shall see how the results derived above might be applied to approximate solutions of the stochastic Poisson equation 
\begin{align*}
    -\Delta u(x) + F(u(x)) = g(x) + \dot{W}(x), \quad x \in D = (0,1)^d, 
\end{align*}
with Dirichlet boundary condition $u(x) = 0$, $x\in \partial D$, and where $F \colon \R \to \R$ is a continuous bounded function satisfying
\begin{equation}\label{lipschitz}
    (x-y)(F(x) - F(y)) \geq -L |x-y|^2, \quad x,y \in \R
\end{equation}
for some \textit{small} constant $L \geq 0$, $g \in L^2(D)$, $\dot{W}$ denotes the formal derivative of the Brownian sheet and $d \in \{2,3\}$.

A \textit{mild} solution to this problem is given by a stochastic process $u = \{u(x) \colon x \in \Bar{D}\}$ verifying, for almost every $x \in D$ and with probability 1, the equation
\begin{equation}\label{Poisson equaiton}
    u(x) + \int_D K(x,y) F(u(y)) dy = \int_D K(x,y)g(y)dy + \int_D K(x,y) W(dy),
\end{equation}
where $K$ is the Green function associated to the Laplace equation with the same boundary conditions, which is given by
\begin{equation}\label{green amb bm}
    K(x,y) = G(x,y) - \E^x\left[ G(B_\tau, y)  \right]
\end{equation}
with
\begin{equation}\label{greendeterminista}
    G(x,y) = \begin{cases}
        \frac{1}{2\pi}\log|x-y|, \quad &d=2, \\
        \frac{1}{4\pi} |x-y|^{-1}, \quad &d=3,
    \end{cases}
\end{equation}
$|\cdot|$ the Euclidean norm in $\R^d$, $B = \{B_t \colon t \geq 0\}$ a standard Brownian motion starting from $x$, $\tau = \inf\{t > 0 \colon B_t \notin D\}$ and $\E^x$ the expectation with respect to the probability measure $\P^x$ under which $B$ is a Brownian motion with these characteristics (see, for instance, \cite{Doob1984} and \cite{Gilbarg2001}). Observe that $K(x,y) = K(y,x)$ and that $K(x,y) = 0$ for almost every $y \in D$ if $x \in \partial D$. Indeed, in such situation one has that $\tau = 0$ $\P^x$-a.s., so $B_{\tau} = x$ $\P^x$-a.s. and therefore, $K(x,y) = G(x,y) - \E^x\left[ G(x, y)  \right] = 0$.

As a consequence of Poincaré's inequality, one has that, for any $\varphi \in L^2(D)$, 
\begin{equation}\label{poincare}
    \left\langle \int_D K(\cdot, y)\varphi(y)dy, \varphi  \right\rangle \geq a \left|\left| \int_D K(\cdot, y) \varphi(y)dy  \right|\right|_2^2,
\end{equation}
for some universal constant $a>0$ (see \cite[Lemma 2.4, p. 225]{pardoux}). It is proven in \cite[Theorem 2.5, p. 225]{pardoux} as well that, whenever $0 \leq L < a$, equation \eqref{Poisson equaiton} has a unique solution with a.s. continuous sample paths.

Our goal is to show that the mild solutions of the Poisson equations
\begin{equation*}
    -\Delta u_n(x) + F(u_n(x)) = g(x) + \theta_n(x),\quad x \in D
\end{equation*}
with same boundary conditions as in \eqref{Poisson equaiton}, which are given by the unique solutions to the integral equations
\begin{equation}\label{Poisson equation aproximadors}
    u(x) + \int_D K(x,y) F(u_n(y)) dy = \int_D K(x,y)g(y)dy + \int_D K(x,y) \theta_n(y)dy,
\end{equation}
converge in law, in the space of continuous functions $\mathcal{C}(D)$ (we write $D$ instead of $\Bar{D}$, the closure of $D$, to simplify the notation, but bear in mind that we are talking about weak convergence of functions defined on $\Bar{D}$) to the mild solution $u$ of \eqref{Poisson equaiton}.

Several things have to be seen in order to prove the desired convergence, for instance, one first has to check that $K(x,\cdot)$ belongs to $L^{2q}(D)$ for some $q \geq 1$ in order to make sense of the integral $\int_D K(x,y)W(dy)$.

Our strategy to prove the convergence will be the following one:
\begin{enumerate}
    \item We will show that $K$ satisfies condition \eqref{L2qHolder} for some $q \geq 1$, $\alpha > 0$. This implies that integrals like $\int_D K(x,y)W(dy)$ are well defined as Wiener integrals and, moreover, by Theorem \ref{main result}, that the integral processes $\int_{D}K(x,y) \theta_n(y)dy$ converge in law, in the space $\mathcal{C}(D)$, towards $\int_D K(x,y)W(dy)$, whenever the kernels $\{\theta_n\}_n$ verify condition \eqref{fita moments}.

    \item We then consider the operator $Tu(x) \coloneqq u(x) + \int_D K(x,y)F(u(y))dy$ and $\eta \in \mathcal{C}_0(D)$ (where $\mathcal{C}_0(D)$ is the space of continuous functions in $\Bar{D}$ vanishing on $\partial D$) and show that:
    \begin{enumerate}[2.1]
        \item if $u \in \mathcal{C}_0(D)$, $Tu$ belongs to $\mathcal{C}_0(D)$ as well,

        \item the equation
        \begin{equation}\label{equacio eta}
            Tu(x) = \int_D K(x,y)g(y) dy + \eta(x)
        \end{equation} 
        has a unique solution $u \in \mathcal{C}_0(D)$. Results concerning the existence and uniqueness of solutions of these equations can be found in, for instance, \cite[Theorem 26.A, p. 557]{Zeidler1990} or \cite[Theorem 2.1, p. 171]{lions1969quelques}, so we will only have to show that our solution is indeed in $\mathcal{C}_0(D)$.
    \end{enumerate}
    \item By the previous item, we will be able to define a map $\Psi \colon \mathcal{C}_0(D) \to \mathcal{C}_0(D)$, $\eta \mapsto \Psi(\eta) = u$, where $u$ is the unique solution in $\mathcal{C}_0(D)$ of \eqref{equacio eta}. Since $K$ vanishes on the boundary and, by the first point, $K$ verifies \eqref{L2qHolder}, we will have that the solutions of \eqref{Poisson equaiton} and \eqref{Poisson equation aproximadors} are given by, respectively,
    \begin{equation*}
        u = \Psi\left( \int_D K(\cdot, y) W(dy)  \right), \quad u_n = \Psi\left( \int_D K(\cdot, y) \theta_n(y)dy  \right).
    \end{equation*}
    So, if we manage to show that the map $\Psi$ is continuous with respect to the uniform norm, the continuous mapping theorem will give us the desired conclusion.
\end{enumerate}

To prove the first point, we follow some of the ideas in \cite[Lemma 2.1]{pardoux} and \cite[Theorem 1]{sanzmartinez}.

\begin{lemma}\label{L2q Holder bounds for K}
    There are positive constants $\alpha, \beta$ and $C$, $\alpha > 2$, such that
    \begin{equation}\label{L2qholderK}
        || K(x,\cdot) - K(z, \cdot ) ||_\alpha \leq C |x - z |^\beta
    \end{equation}
    for any $x,z \in \Bar{D}$.
\end{lemma}

\begin{proof}
    We first see that, if we manage to show that
    \begin{equation}\label{Lipschitz Lalfa G}
        || G(x,\cdot) - G(z, \cdot ) ||_\alpha \leq C |x - z |^\beta,
    \end{equation}
    then we will have \eqref{L2qholderK} as well. Indeed by the strong Markov property, for any $x,y,z \in \Bar{D}$, 
    \begin{equation*}
        \E^z\left[ G(B_\tau, y)  \right] = \E^x\left[ G(B_\tau - x + z, y)  \right],
    \end{equation*}
    so
    \begin{align*}
        || K(x,\cdot) - K(z, \cdot ) ||_\alpha &\leq || G(x, \cdot) - G(z,\cdot) ||_\alpha + ||\, \E^x\left[ G(B_\tau, \cdot)\right] - \E^x\left[ G(B_\tau - x + z, \cdot)  \right] ||_\alpha   \\
        &\leq C |x-z|^\beta + \left(\int_D\left|  \int_{\Omega} \left[ G(u, y) - G(u - x + z, y) \right]d\P^x(u) \right|^{\alpha}dy  \right)^{\frac{1}{\alpha}}.
    \end{align*}
    By Jensen's inequality and Tonelli's Theorem, we have
    \begin{align*}
        \int_D\left|  \int_{\Omega} \left[ G(u, y) - G(u - x + z, y) \right]d\P^x(u) \right|^{\alpha}dy &\leq \int_D \left(\int_{\Omega}  \left| G(u, y) - G(u - x + z, y) \right|^\alpha d\P^x(u) \right) dy  \\
        &= \int_{\Omega} || G(u, \cdot) - G(u - x + z, \cdot) ||_\alpha^\alpha d\P^x(u)\\
        &\leq C^{\alpha } |x - z|^{\alpha \beta}.
    \end{align*}
    So, all in all,
    \begin{equation*}
        || K(x,\cdot) - K(z, \cdot ) ||_\alpha \leq 2C |x-z|^{\beta}.
    \end{equation*}
    Proving that it suffices to show \eqref{Lipschitz Lalfa G}.

    \textbf{Case $d = 2$}

    Given $\alpha > 2$ and $0< \varepsilon <  \alpha$ (other restrictions might appear on $\varepsilon$ during the process), Hölder's inequality and the mean value theorem applied to the function $\log x$ yield
    \begin{align*}
        \int_D& \left| \log|x-y| - \log|z-y|  \right|^\alpha dy = \int_D \left| \log|x-y| - \log|z-y|  \right|^\varepsilon \left| \log|x-y| - \log|z-y|  \right|^{\alpha -\varepsilon}  dy \\
        &\leq \left( \int_D \left| \log|x-y| - \log|z-y|  \right|^{\varepsilon q} dy \right)^{\frac{1}{q}} \left( \int_D \left| \log|x-y| - \log|z-y|  \right|^{(\alpha -\varepsilon)p}\right)^{\frac{1}{p}} \\
        &= \left( \int_D \left| \log|x-y| - \log|z-y|  \right|^{\varepsilon q} dy \right)^{\frac{1}{q}} \left( \int_D \left| \frac{|x-y| - |z-y|}{\theta(|x-y|-|z-y|) + |z-y|}  \right|^{(\alpha -\varepsilon)p}\right)^{\frac{1}{p}} \\
        &\leq |x-z|^{\alpha - \varepsilon} \left( \int_D \left| \log|x-y| - \log|z-y|  \right|^{\varepsilon q} dy \right)^{\frac{1}{q}} \left( \int_D  \frac{1}{|x-y|^{(\alpha -\varepsilon)p}}dy + \int_D  \frac{1}{|z-y|^{(\alpha -\varepsilon)p}}dy  \right)^{\frac{1}{p}},
    \end{align*}
    for some $\theta \in (0,1)$, $p,q > 1$ such that $1/p + 1/q = 1$ and where we have used that $||a| - |b|| \leq |a-b|$ for any $a,b \in \R$ and that the maximum of $\left| \theta(|x-y|-|z-y|) + |z-y| \right|^{-(\alpha -\varepsilon)p}$ is attained at $\theta = 0$ or $\theta = 1$. Using, as well, that $|a + b|^q \leq C_q (|a|^q + |b|^q)$ for any $a,b \in \R$ and $q \geq 0$, we see that this last expression can be bounded by
    \begin{equation*}
        C|x-z|^{\alpha - \varepsilon} \left( \int_D |\log|x-y||^{\varepsilon q}dy + \int_D |\log|z-y||^{\varepsilon q}dy  \right)^{\frac{1}{q}} \left( \int_D  \frac{1}{|x-y|^{(\alpha -\varepsilon)p}}dy + \int_D  \frac{1}{|z-y|^{(\alpha -\varepsilon)p}}dy  \right)^{\frac{1}{p}}.
    \end{equation*}
    If we let $B(\xi,\rho)$ denote the ball (disk in two dimensions) of radius $\rho > 0$ and center $\xi \in \R^d$, then we have $D \subset B(x,2)$, implying that
    \begin{align*}
        \int_D |\log|x-y||^{\varepsilon q}dy \leq \int_{B(x,2)} |\log|x-y||^{\varepsilon q}dy = 2\pi \int_0^2  |\log r|^{\varepsilon q} r dr = 2\pi \int_{-\infty}^{\log 2} |t|^{\varepsilon q } e^{2t} dt < \infty.
    \end{align*}
    Where we have made a change to polar coordinates (centered at $x$) and applied the substitution $t = \log r$. We see that the obtained bound (finite for any $\varepsilon> 0$ and $q > 1$), does not depend on $x$. Similarly, we can bound the integral $\int_D |\log|z-y||^{\varepsilon q}dy$ by a finite constant independent of $z$. We can see, as well, that
    \begin{equation*}
        \int_D  \frac{1}{|x-y|^{(\alpha -\varepsilon)p}}dy \leq \int_{B(x,2)}  \frac{1}{|x-y|^{(\alpha -\varepsilon)p}}dy = 2\pi \int_0^2 \frac{r}{r^{(\alpha -\varepsilon)p}}dr.
    \end{equation*}
    This last integral is finite if $(\alpha - \varepsilon)p < 2$, or, equivalently, if $\alpha - \frac{2}{p} < \varepsilon $.
    
    Hence, we have proven that, for all $\alpha > 2$, we can find $p > 1$ with $0 < \alpha - 2/p$, $\varepsilon \in (\alpha - 2/p,\alpha)$ and a constant $C$ (which will depend on this last three parameters) such that 
    \begin{equation*}
        ||G(x,\cdot) - G(z,\cdot)||_\alpha \leq C |x-z|^{\frac{\alpha - \varepsilon}{\alpha}}
    \end{equation*}
    for any $x,z \in D$ in the case $d = 2$.

    \textbf{Case $d = 3$}

    Again, let us fix $0 < \varepsilon < \alpha$. By Hölder's inequality and the mean value theorem (applied to the function $1/x$), we have
    \begin{align*}
        \int_D& \left| \frac{1}{|x-y|} - \frac{1}{|z-y|}  \right|^{\alpha} dy = \int_D \left| \frac{1}{|x-y|} - \frac{1}{|z-y|}  \right|^{\varepsilon} \left| \frac{1}{|x-y|} - \frac{1}{|z-y|}  \right|^{\alpha - \varepsilon}dy \\
        &\leq C |x-z|^{\alpha - \varepsilon} \left( \int_D \frac{dy}{|x-y|^{\varepsilon q}} + \int_D \frac{dy}{|z-y|^{\varepsilon q}} \right)^{\frac{1}{q}} \left( \int_D\frac{dy}{\left| \theta(|x-y|-|z-y|) + |z-y| \right|^{2(\alpha -\varepsilon)p}}\right)^{\frac{1}{p}} \\
        &\leq C |x-z|^{\alpha - \varepsilon} \left( \int_D \frac{dy}{|x-y|^{\varepsilon q}} + \int_D \frac{dy}{|z-y|^{\varepsilon q}} \right)^{\frac{1}{q}} \left( \int_D\frac{dy}{|x-y| ^{2(\alpha -\varepsilon)p}} + \int_D\frac{dy}{|z-y| ^{2(\alpha -\varepsilon)p}}\right)^{\frac{1}{p}}
    \end{align*}
    where $\theta \in (0,1)$, $p,q>1$ are such that $1/p + 1/q = 1$ and in the last step we have used that the maximum of $\left| \theta(|x-y|-|z-y|) + |z-y| \right|^{-2(\alpha -\varepsilon)p}$ is attained at $\theta= 0$ or $\theta = 1$. 
    Similar arguments to the ones already seen for the case $d = 2$ (bearing in mind that now we must perform a spherical change of coordinates, hence, a factor $r^2$ will appear from the Jacobian) lead to
    \begin{equation*}
        \int_D \frac{dy}{|x-y|^{\varepsilon q}} \leq C \int_0^2 \frac{r^2}{r^{\varepsilon q}} dr
    \end{equation*}
    where $C> 0$ is some constant independent of $x$. This last integral converges if $\varepsilon < \frac{3}{q} = \frac{3(p-1)}{p}$. On the other side,
    \begin{equation*}
        \int_D\frac{dy}{|x-y| ^{2(\alpha -\varepsilon)p}} \leq C \int_0^2 \frac{r^2}{r^{2(\alpha -\varepsilon)p}}dr
    \end{equation*}
    for some constant $C> 0$ which is, again, independent of $x$. This last integral converges if $\varepsilon > \alpha - \frac{3}{2p}$ (the same arguments can be used to bound the integrals depending on $z$). Hence, $\varepsilon \in (\alpha - \frac{3}{2p}, \frac{3(p-1)}{p})$ is required. Contrary to what we have seen for $d = 2$, a new upper bound for $\varepsilon$ (apart from $\varepsilon < \alpha$) appears, which, in turn, gives a restriction on the parameter $\alpha$, which is that $\alpha - \frac{3}{2p} < \frac{3(p-1)}{p}$ or, equivalently
    \begin{equation*}
        \frac{\alpha}{3}< \frac{p-1}{p} + \frac{1}{2p} = \frac{2p-1}{2p} < 1.
    \end{equation*}
    Hence, for $d = 3$, we can only show condition \eqref{Lipschitz Lalfa G} for $\alpha < 3$. Let us fix, for instance, $p = 2$, then we have to determine for which values of $\alpha$ the intersection $(-\infty, 3/2) \cap (\alpha - 3/4, \infty)$ is non-empty, which happens to be true for $\alpha < 9/4$. That is, for each $2 < \alpha < 9/4$, we can find $\varepsilon \in (-\infty, 3/2) \cap (\alpha - 3/4, \infty)$ so that \eqref{Lipschitz Lalfa G} holds with $\beta = \frac{\alpha - \varepsilon}{\alpha}$ and some constant $C > 0$ depending on $\alpha$ and $\varepsilon$, but not on $x$ nor $z$.
\end{proof}

A different (and much simpler) proof of this result can be given when $\alpha = 2$ is considered in \eqref{L2qholderK}. This proof relies on the fact that, for $D = (0,1)^d$, the Green function $K$ admits a rather simpler and manageable representation in terms of Fourier series. We refer to \cite[Lemma 3.2]{istvan} for a proof of this result in this case. The main reason to prove Lemma \ref{L2q Holder bounds for K} is that, as seen in \cite[Section 3.2]{bardinaboukfal1}, hypothesis \eqref{fita moments} has only been proved for $q > 1$ when the Kac-Stroock kernels are considered. 

We now address the problem of existence and uniqueness of solutions in $\mathcal{C}_0(D)$ of equation \eqref{equacio eta}. 

\begin{proposition}
    The map $T \colon \mathcal{C}_0(D) \to \mathcal{C}_0(D)$ defined by
    \begin{equation*}
        Tu(x) = u(x) + \int_D K(x,y)F(u(y))dy
    \end{equation*}
    is bijective. That is, for any $b \in \mathcal{C}_0(D)$, there is a unique $u \in \mathcal{C}_0(D)$ such that $Tu = b$.
\end{proposition}

\begin{proof}
    We first see that, if $u \in \mathcal{C}_0(D)$, then $Tu$ is in the same space. If $x \in \partial D$, then, since $K$ vanishes on $\partial D$, $Tu(x) = u(x) = 0$. Moreover, for $x,z \in \Bar{D}$, 
    \begin{equation*}
        \left| Tu(x) - Tu(z)  \right| \leq |u(x) - u(z)| + C|| K(x,\cdot) - K(z,\cdot)||_2,
    \end{equation*}
    where we have used that $F$ is bounded and Cauchy-Schwarz's inequality. Appealing to Lemma \ref{L2q Holder bounds for K} and the fact that $u$ is continuous, we obtain the continuity of $Tu$.

    As previously mentioned, the existence and uniqueness of solutions is a consequence of \cite[Theorem 26.A, p. 557]{Zeidler1990} or \cite[Theorem 2.1, p. 171]{lions1969quelques}. The computations to check that the hypotheses of these results are fulfilled (that is, that $T$ is strictly monotone, coercive and hemicontinuous as a map from $W_0^{1,2}(D)$, the usual Sobolev space of functions with null trace, to its topological dual, $(W_0^{1,2}(D))^*$) can be found in \cite[Theorem 2.2]{SANZSOLE20181857}, so we will omit them. 
    
    What we will show is that, if $b \in \mathcal{C}_0(D)$, then the solution is in the same space as well. Indeed, since $u$ solves $Tu = b$, we have that, for any $x,z \in \Bar{D}$,
    \begin{equation*}
        |u(x) - u(z)| \leq C ||K(x,\cdot) - K(z,\cdot)||_2 + |b(x) - b(z)|,
    \end{equation*}
    since $F$ is bounded. Again, the continuity of $b$ and Lemma \ref{L2q Holder bounds for K} imply the continuity of $u$. Moreover, we have that $u$ vanishes on $\partial D$ since $\int_D K(x,y)F(u(y))dy$ and $b(x)$ do so.
\end{proof}

Finally, we show that the map $\Psi$ is continuous with respect to the uniform norm.

\begin{proposition}
    The map $\Psi \colon \mathcal{C}_0(D) \to \mathcal{C}_0(D)$, $\eta \mapsto \Psi(\eta) = u$, where $u \in \mathcal{C}_0(D)$ is the unique solution of \eqref{equacio eta} is continuous with respect to the uniform norm.
\end{proposition}
\begin{proof}
    For each $x \in \Bar{D}$ and $\eta, \eta' \in \mathcal{C}_0(D)$, we have
    \begin{align*}
        |\Psi(\eta)(x) - \Psi(\eta')(x)| &\leq \int_D |K(x,y)||F(\Psi(\eta(y)) - F(\Psi(\eta'(y))|dy + |\eta(x) - \eta'(x)| \\
        &\leq ||K(x,\cdot)||_2 ||F(\Psi(\eta)) - F(\Psi(\eta'))||_2 + ||\eta - \eta'||_\infty
    \end{align*}
    where $||\cdot||_\infty$ is the uniform norm on $D$. Hence,
    \begin{equation}\label{quasi continu}
        || \Psi(\eta) - \Psi(\eta')||_\infty \leq \Lambda ||F(\Psi(\eta)) - F(\Psi(\eta'))||_2 + ||\eta - \eta'||_\infty,
    \end{equation}
    where $\Lambda \coloneqq \sup_{x \in D} ||K(x,\cdot)||_2$ (which is finite by Lemma \ref{L2q Holder bounds for K}). Now let $v_n, v \in L^2(D)$ be such that $v_n \to v$ in $L^2(D)$, then we have that $F(v_n) \to F(v)$ in $L^2(D)$ as well. Indeed, given that $\{v_n\}_n$ converges in $L^2(D)$ to $v$, any subsequence $\{v_{n_k}\}_k$ will converge $L^2(D)$ towards $v$ as well. This will imply the existence of a further subsequence $\{v_{n_{k_j}}\}_j$ converging to $v$ almost everywhere in $D$. Since $F$ is continuous, $F(v_{n_{k_j}})$ will converge to $F(v)$ almost everywhere in $D$ and, since it is bounded, the dominated convergence theorem yields
    \begin{equation*}
        ||F(v_{n_{k_j}}) - F(v)||_2 \xrightarrow{j \to \infty}0.
    \end{equation*}
    Proving that any subsequence of $\{F(v_n)\}_n$ has a further subsequence converging, in $L^2(D)$, towards $F(v)$, which shows that $\{F(v_n)\}_n$ converges towards $F(v)$ in $L^2(D)$.

    Next, given a sequence $\{\eta_n\}_n \subset \mathcal{C}(D)$ uniformly convergent towards $\eta \in \mathcal{C}(D)$, then we have $||\eta - \eta_n||_2 \to 0$ as well, so, by what we have just seen, and by \eqref{quasi continu}, it suffices to show that $\Psi(\eta_n) \to \Psi(\eta)$ in $L^2(D)$ to show that $\Psi(\eta_n)$ uniformly converges towards $\Psi(\eta)$.

    To show the convergence in $L^2(D)$, we will follow some of the computations in \cite[Theorem 2.5]{pardoux} or \cite[Theorem 2.2]{SANZSOLE20181857} used to prove the existence of solutions to equations like $Tu = b$.

    Let $u = \Psi(\eta)$ and $u_n = \Psi(\eta_n)$, then, for each $x \in \Bar{D}$,
    \begin{equation}\label{diferencia solucions}
        \eta(x) - \eta_n(x) = u(x) - u_n(x) + \int_D K(x,y)\left[ F(u(y)) - F(u_n(y)) \right]dy.
    \end{equation}
    Taking the inner product of this last expression against $F(u) - F(u_n)$ and using \eqref{poincare} (with $\varphi = F(u) - F(u_n)$) and \eqref{lipschitz} (with $x = u$ and $y = u_n$) gives
    \begin{align}\label{expressio1}
        \langle F(u) - F(u_n), \eta - \eta_n \rangle 
        \geq -L ||u - u_n||_2^2 + a \left| \left| \int_D K(\cdot,y)\left[ F(u(y)) - F(u_n(y)) \right]dy \right| \right|_2^2.
    \end{align}
    On the other hand, we also have, by \eqref{diferencia solucions} again,
    \begin{equation*}
        \left| \left| \int_D K(\cdot,y)\left[ F(u(y)) - F(u_n(y)) \right]dy \right| \right|_2^2 = ||u - u_n||_2^2 + ||\eta - \eta_n||_2^2 - 2\langle u - u_n, \eta - \eta_n\rangle.
    \end{equation*}
    Plugging this in \eqref{expressio1} yields
    \begin{align}\label{més fites}
        a||\eta - \eta_n||_2^2 + (a-L)||u - u_n||_2^2 &\leq \langle F(u) - F(u_n) + 2a(u - u_n), \eta - \eta_n \rangle
        \nonumber\\
        &\leq ||\eta - \eta_n ||_2 \left(  ||F(u) - F(u_n)||_2 + 2a||u- u_n||_2 \right),
    \end{align}
    where in the last step we have used Cauchy-Schwarz's inequality and the triangle inequality. Finally, using the fact that
    \begin{equation*}
        ||u - u_n ||_2 \leq \Lambda ||F(u) - F(u_n) ||_2 + ||\eta - \eta_n||_2,
    \end{equation*}
    we deduce, from \eqref{més fites}, that
    \begin{equation*}
        a||\eta - \eta_n||_2^2 + (a-L)||u - u_n||_2^2 \leq ||\eta - \eta_n ||_2 \left( (1+2a \Lambda) ||F(u) - F(u_n)||_2 + 2a||\eta- \eta_n||_2 \right).
    \end{equation*}
    Since $F$ is bounded and $L < a$, we obtain that $|| u - u_n||_2 \to 0$ as $||\eta - \eta_n||_2 \to 0$, as was to be shown.
\end{proof}

The results seen in this section hold true if the assumption on $F$ being bounded is dropped but we assume that $F$ is a Lipschitz function whose Lipschitz constant, $L$, verifies the condition $L < \Lambda^{-1}$, where $\Lambda \coloneqq \sup_{x \in D} ||K(x,\cdot)||_2$. The existence and uniqueness of solutions to \eqref{equacio eta} then follows from Banach's fixed point theorem, while the continuity of the map $\Psi$ follows from a rather simple and straightforward computation.

\begin{proposition}
    Suppose $F$ is $L$-Lipschitz with $L < \Lambda^{-1}$, then, for each $\eta \in \mathcal{C}_0(D)$, equation \eqref{equacio eta} has a unique solution in $\mathcal{C}_0(D)$ and the induced map $\Psi \colon \mathcal{C}_0(D) \to \mathcal{C}_0(D)$ is continuous with respect to the uniform norm.
\end{proposition}
\begin{proof}
    Given $\eta \in \mathcal{C}_0(D)$, consider the operator $\Tilde{T} \colon \mathcal{C}_0(D) \to \mathcal{C}_0(D)$ defined by
    \begin{equation*}
        \Tilde{T}u(x) = - \int_D K(x,y)F(u(y))dy + \int_D K(x,y)g(y)dy + \eta(x), \quad x \in \Bar{D}.
    \end{equation*}
    The fact that $\Tilde{T}u \in \mathcal{C}_0(D)$ if $u \in \mathcal{C}_0(D)$ follows from the fact that $\int_D K(\cdot, y)F(u(y))dy$, $\int_D K(\cdot, y)g(y)dy$ and $\eta$ vanish on $\partial D$ and, for all $x,z \in \Bar{D}$,
    \begin{align*}
        \left| \Tilde{T}u(x) - \Tilde{T}u(z)  \right| \leq  ||K(x,\cdot) - K(z,\cdot)||_2 \left( |F(0)| + L||u||_\infty + ||g||_2  \right)  + |\eta(x) - \eta(z)|,
    \end{align*}
    so, by Lemma \ref{L2q Holder bounds for K} and the fact that $\eta$ is continuous, we see that $\Tilde{T}u$ is continuous.

    Now, given $u, u' \in \mathcal{C}_0(D)$, it is easy to see that
    \begin{equation*}
        ||\Tilde{T}u - \Tilde{T}u'|| \leq \Lambda L ||u - u'||_\infty
    \end{equation*}
    which, since $\Lambda L < 1$, tells us that $\Tilde{T}$ is a contraction for fixed $\eta \in \mathcal{C}_0(D)$. Given that $\mathcal{C}_0(D)$ endowed with the uniform norm is a Banach space, we obtain the existence of a unique fixed point in $\mathcal{C}_0(D)$ for the map $\Tilde{T}$, proving the existence and uniqueness of solutions of \eqref{equacio eta}.

    Lastly, consider $\eta, \eta' \in \mathcal{C}_0(D)$, then it is not so hard to see that
    \begin{equation*}
        ||\Psi(\eta) - \Psi(\eta') ||_\infty \leq \Lambda L ||\Psi(\eta) - \Psi(\eta') ||_2 + ||\eta - \eta' ||_\infty \leq \Lambda L ||\Psi(\eta) - \Psi(\eta') ||_\infty+ ||\eta - \eta' ||_\infty
    \end{equation*}
    which implies, using that $1 - \Lambda L > 0$,
    \begin{equation*}
        ||\Psi(\eta) - \Psi(\eta') ||_\infty \leq \frac{1}{1 - \Lambda L} ||\eta - \eta'||_\infty,
    \end{equation*}
    and thus, that $\Psi$ is a continuous map, finishing the proof.
\end{proof}

\end{document}